\documentclass[12pt]{article}
 \makeatletter
\newfont{\footsc}{cmcsc10 at 8truept}
\newfont{\footbf}{cmbx10 at 8truept}
\newfont{\footrm}{cmr10 at 10truept}
\renewcommand{\ps@plain}{%
\renewcommand{\@oddfoot}{\footsc
  {\footbf }  \footrm\thepage}}
\makeatother  \leftmargin=25mm

\usepackage{arxiv}

\usepackage[utf8]{inputenc} 
\usepackage[T1]{fontenc}    
\usepackage{hyperref}       
\usepackage{url}            
\usepackage{booktabs}       
\usepackage{amsfonts}       
\usepackage{nicefrac}       
\usepackage{microtype}      
\usepackage{lipsum}
\usepackage[centertags]{amsmath}
\usepackage{amsfonts}
\usepackage{amsthm}
\usepackage{newlfont}
\usepackage{graphicx}
\usepackage{subfigure}
\usepackage{caption}
\usepackage{amsfonts}
\usepackage{amssymb}
\usepackage{mathrsfs}
\usepackage{geometry}

\newtheorem{conj}{Conjecture}

\newtheorem{thm}{Theorem}[section]
\newtheorem{cor}[thm]{Corollary}
\newtheorem{lem}[thm]{Lemma}

\theoremstyle{definition}
\newtheorem{defn}[thm]{Definition}
\theoremstyle{remark}
\newtheorem{rem}[thm]{Remark}

\title{Covering the crosspolytope with its smaller
homothetic copies}   
\date{}
\author{Yanlu Lian$^1$ and Yuqin Zhang$^2${\footnote{Corresponding
author. }} \thanks{E-mail addresses: yanlu\_lian@tju.edu.cn, yuqinzhang@tju.edu.cn.}}

\begin{document}
\maketitle
\begin{center}\vskip -0.8cm
{\small $^1$Center for Applied Mathematics,
Tianjin University, 300072, Tianjin, China \\$^2$ School of Mathematics, Tianjin University, Tianjin, 300072, China
 \vskip 0.3cm
}
\end{center}

\begin{abstract}
  In 1957, Hadwiger made the famous conjecture that any convex body of $n$-dimensional Euclidean space $\mathbb{E}^n$ can be covered by $2^n$ smaller positive homothetic copies. Up to now, this conjecture is still open for all $n\geq 3$. Denote by $\gamma_{m}(K)$ the smallest positive number $\lambda$ such that $K$ can be covered by $m$ translations of $\lambda K$. The values of $\gamma_m(K)$ for some particular $m$ and $K$ have been studied. In this article, we will focus on the situation where $K$ is the unit crosspolytope of the three-dimensional.



\end{abstract}

\maketitle

\section{ Introduction and preliminaries}

 \quad Let $\mathcal{K}^n$ be the set of all convex bodies, i.e., compact convex sets, in the $n$-dimensional Euclidean space $\mathbb{E}^n$ with non-empty interior. The boundary of a convex body $K$ is denoted by $\partial(K)$ and the interior is denoted by $int(K)$. Let $c(K)$ denote the covering number of $K,$ i.e., the smallest number of translations of $\mathrm{int}(K)$ such that their union can cover $K$.
 For arbitrarily two points $\boldsymbol x, \boldsymbol y$ in $\mathbb{E}^n$, let $xy$ denote the closed segment joining points $\boldsymbol x$ and $\boldsymbol y$, and let $\left\vert{xy}\right\vert$ be the distance between $\boldsymbol x$ and $\boldsymbol y$. By $L_{xy}$ we denote the straight line through disjoint points $\boldsymbol x$ and $\boldsymbol y$. For a convex body, denote by $d(K)$ the diameter of $K$, i.e.,
 $$d(K)=\sup\{{\left\vert{xy}\right\vert}: \boldsymbol x, \boldsymbol y\in K\}.$$

 In 1955, Levi \cite{Levi} studied the covering number in dimension $2$. He proved that
 $$c(K)=
 \begin{cases}
 4,&\text{if K is a parallelogram},\\
 3,&\text{otherwise.}
 \end{cases}$$
 In 1957, Hadwiger \cite{Hadwiger} studied this number further and proposed the famous conjecture:
 \begin{conj}
 {\bf{(Hadwiger's covering conjecture)}} For every $K\in \mathcal{K}^n$ we have
 $$c(K)\leq2^n,$$
 where the equality holds if and only if $K$ is a parallelopiped.
\end{conj}
 This conjecture has been studied by many authors. Lassak \cite{Lassak1} proved this conjecture for the three-dimensional centrally symmetric case. Rogers and Zong \cite{Rogers} obtained the currently best known upper bound
 $$c(K)\leq \binom{2n}{n}(nlogn+nloglogn+5n)$$
 for general $n$-dimensional convex bodies, and
 $$c(K)\leq2^n(nlogn+nloglogn+5n)$$
 for centrally symmetric ones. Combining ideas from \cite{Artstein} with a new result on the K{\"o}vner-Besicovitch measure of symmetry for convex bodies, Huang et al. \cite{Huang} obtained a new general upper bound for Hadwiger's problem: there exist universal constants $c_1, c_2 > 0$ such that for all $n\geq 2$ and every convex body $K \in \mathcal{K}^n$,
 $$c(K)\leq c_14^ne^{c_2\sqrt{n}}.$$
 Moreover, two different versions have been discovered and studied (see \cite{Boltyanski1} \cite{Boltyanski2} \cite{Brass} \cite{Zong2} for general references). Nevertheless, we are still far away from the solution of the conjecture, even three-dimensional case.

 For $K \in \mathcal{K}^n$ and any positive integer $m$, we denote by $\gamma_m(K)$ the smallest positive number $\lambda$ such that $K$ can be covered by $m$ translations of $\lambda K$, i.e.,
  $$\gamma_m(K)=\min\{\lambda>0: \exists\{\boldsymbol u_i:i=1,\dots,m\}\subset \mathbb{E}^n, s.t.\quad K\subseteq \cup_{i=1}^{m}(\lambda K+\boldsymbol u_i)\}.$$
Clearly, we have $\gamma_m(K)=1$, for all $m\leq n$, and $\gamma_m(K)$ is a non-increasing step sequence for all positive integers $m$ and all $n$-dimensional convex bodies $K$.

 It is easy to see that $c(K)\leq m$ for some $m\in Z^{+}$ if and only if $\gamma_m(K) < 1$. Thus, $\gamma_m(K)$ plays an important role in the Hadwiger's covering conjecture. Lassak \cite{Lassak2} showed that for every two-dimensional convex domain $K$,
 $$\gamma_4(K)\leq \frac{\sqrt{2}}{2}.$$
Zong \cite{Zong2} obtained
$$\gamma_8(C)\leq \frac{2}{3}$$
for a bounded three-dimensional convex cone $C$, and
$$\gamma_8(K_p)\leq \sqrt{\frac{2}{3}}$$
for all the unit ball $K_p$ of the three-dimensional $l_p$ spaces.

 In this paper, we will only deal with the situation where $K_1$ is the unit cross-polytope of the three-dimensional, in other words,
  $$K_1=\{(x_1,x_2,x_3): |x_1|+|x_2|+|x_3|\leq 1\}.$$
 Moreover, for simplicity of notation, we let $\Delta $ denote an equilateral triangle, and instead of $\gamma_m(K_1)K_1$ we write $\gamma_mK_1$. We say that the three planes of symmetry of $K_1$ are base squares. Other notations are denoted as in Figure 1.
 \begin{figure}[htbp]
\centering
\includegraphics[scale=0.4]{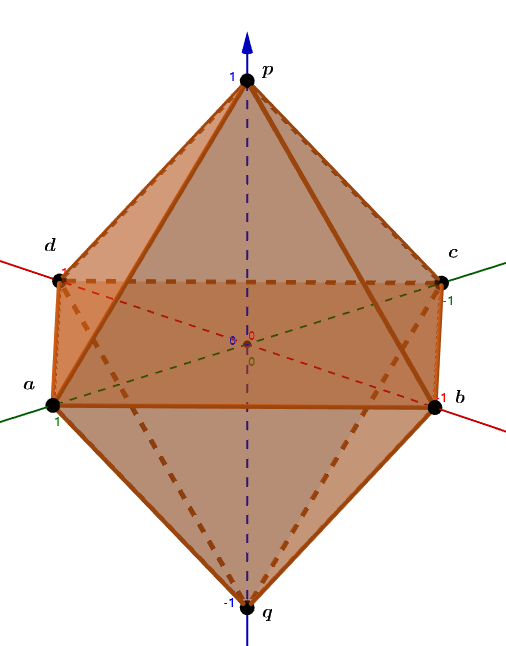}
\caption{$K_1$}
\end{figure}

As we know, $\gamma_m(K_1)=1$ for all $m\leq n=3$, but for $m>3$ we obtain the following results.
\begin{thm}
For $m=4,5$, $\gamma_m(K_1)=1$.
\end{thm}

When $m=6$ is equal to the number of vertices of $K_1$, we find that the intersections of a homothetic copy of $K_1$ which contains a vertex $\boldsymbol v$ of $K_1$ with facets associate to $\boldsymbol v$ are contained in some equilateral triangles, then together with $\gamma_3(T)$ for an equilateral triangle $T$ in two-dimensional, we determine the theorem as follows.
\begin{thm}
For $m=6$, $\gamma_m(K_1)=\frac{2}{3}.$
\end{thm}

when $m>6$, except the six homothetic copies that cover the six vertices of $K_1$, the last $m-6$ homothetic copies have to cover the parts that are not be covered. Here we introduce an interesting point that we need in the following theorems, we name it $\lambda$-node point. Given an equilateral triangle $T$, and let $T_1, T_2, T_3$ be homothetic copies of $T$ with ratio $\lambda\geq \frac{1}{2}$. These three homothetic copies of $T$ contained in $T$ and each of them has one vertex coinciding with a vertex of $K_1.$ Then we say that the intersection points of the sides of the three homothetic triangles are $\lambda$-node points of $T$. Specially, when $\lambda=\frac{2}{3}$, these three intersection points are conincide, we call this point the centre of $T$.
\begin{thm}
For $m=7, 8,9$, $\gamma_m(K_1)=\frac{2}{3}.$
\end{thm}

When $\gamma_m(K_1)$ is less than some positive number $\lambda$, it is easy to see that the $\lambda$-node points in the facets of $K_1$ can not be covered by the homothetic copies of $K_1$ that contain one vertex of $K_1$. Then we have the following theorems.
\begin{thm}
For $m=10$, $\gamma_m(K_1)=\frac{3}{5}.$
\end{thm}
As a consequence of Theorem 1.4, together with the Pigeonhole principle, we immediately get the following theorem.
\begin{thm}
For $m=11,12,13$, $\gamma_m(K_1)=\frac{3}{5}.$
\end{thm}

Since $\gamma_m(K_1)$ is a non-increasing step sequence for all positive integers $m$, the distance between two $\gamma_m(K_1)$ is getting bigger. Then, the $\gamma_m(K_1)$-node points in one facet can not be covered by one homothetic copy of $K_1$, Thus, we get that
\begin{thm}
For $m=14$, $\gamma_m(K_1)=\frac{4}{7}.$
\end{thm}

\begin{thm}
For $m=15,16,17$, $\gamma_m(K_1)=\frac{4}{7}.$
\end{thm}

This paper is organized as follows. In Section 2, we show the proof for $m=4,5$; In Section 3, firstly, we will show a lemma that plays an important role in the following theorem proofs. Next, we will show the proofs for $m=6,7,8,9$; In Section 4, we are going to prove the case for $m=10,11,12,13$;
In section 5, we are going to prove the cases for $m=14,15,16,17$.  In the proofs convexity methods are combined
with combinatorial properties of geometry.

\section{Case for $m=4,5$. }

\begin{lem}\label{070327prop1}
 For a positive number $\lambda$, if there exists a translation of $\lambda K_1$ cover two vertices from $K_1$, then $\lambda\geq1$.
\end{lem}

\begin{proof}
 Let $V=\{\boldsymbol p, \boldsymbol a, \boldsymbol b, \boldsymbol c, \boldsymbol d, \boldsymbol q\}$ be the vertices set of $K_1$ and let $\boldsymbol x$ be a point in $\mathbb{E}^3$. Without loss of generality, we may assume that $\boldsymbol v, \boldsymbol w$ are two vertices of $K_1$ from $V$ that are contained in the translation $\lambda K_1+\boldsymbol x$. By the convexity of $K_1$, the segment $vw$ is covered by $\lambda K_1+\boldsymbol x$.

 {\bf Case 1:} $vw$ is a diagonal of base square $S$. Since $\lambda K_1+\boldsymbol x$ is a translation of $\lambda K_1$, $vw // \lambda vw+\boldsymbol x$ and $vw\subset \lambda K_1+\boldsymbol x$. Therefore, $$|vw|\leq d(\lambda K_1+\boldsymbol x)=\lambda d(K_1)=\lambda|vw|,$$ then $\lambda\geq 1$.

 {\bf Case 2:} $vw$ is an edge of base square $S$. Let $\pi$ be a plane containing segment $vw$ and parallel to base square $S$. And let $S^{'}=\pi \cap (\lambda K_{1}+\boldsymbol x)\supset vw$. Because $\lambda K_1+\boldsymbol x$ is a translation of $\lambda K_1$ and $S//S^{'}$, $Q^{'}=\lambda S+\boldsymbol x$. $S$ is a two-dimensional square, so is $S^{'}$. We know that $vw$ is an edge of the square $S$, and we can assume that $l$ is the edge of $S^{'}$. Then we can get that $l=\lambda vw+\boldsymbol x$. Because $vw$ is also in $S^{'}$,
 $$|vw|\leq |\lambda vw+x|=\lambda|vw|.$$
 We can see  $\lambda\geq1$.
\end{proof}

\begin{proof}[Proof of Theorem $1.1$]
We assume that $K_1 \subset \bigcup_{i=1}^m \gamma_m(K_1)+\boldsymbol u_i$, and $\boldsymbol u_i\in \mathbb{E}^3,i=1,\dots,m$. Since $K_1$ has six vertices, and $m=4$, $5<6$. That is to say, the number of translations of $\lambda K_1$ is less than the number of vertices of $K_1$. By the Pigeonhole principle, there is always a point $\boldsymbol u_i, i\in \{1,2,\dots,m\}$ such that $\gamma_mK_1+\boldsymbol u_i$ covers two vertices of $K_1$. By Lemma 2.1, we can see that $\gamma_m(K_1)=1$.
\end{proof}

\section{Case for $m=6,7,8,9$. }

\begin{defn}
Let $X=\{\boldsymbol u_1, \boldsymbol u_2, \dots, \boldsymbol u_m\}\subset \mathbb{E}^n$. If $\gamma_mK_1+X$ are the homothetic copies of $K_1$ such that
$$K\subseteq \bigcup_{i=1}^m \gamma_mK_1+\boldsymbol u_i,$$
then we call $\gamma_mK+X$ an optimal $m$-covering configuration of $K_1$.
\end{defn}
The following observation is crucial to our work:
\begin{flushleft}
 {\bf Observation :} To construct an optimal covering configuration, $X$ should be contained in $K_1$, that is $\boldsymbol u_i \in K_1$, $i=1, \dots, m$. Let $\lambda K_1+\boldsymbol u$ be the translation of $\lambda K_1$ which covers one of vertices of $K_1$. Without loss of generality, we assume that $\boldsymbol p\in \lambda K_1+\boldsymbol u$. By some computation, there are three possibilities:
 \end{flushleft}

Case 1. $\boldsymbol u\in pq$ or $ac$ or $bd$. The intersections of four facets of $K_1$ which associate to $\boldsymbol p$ with $\lambda K_1+\boldsymbol u$ are four triangles.

Case 2.  $\boldsymbol u\in conv\{\boldsymbol p, \boldsymbol a, \boldsymbol q, \boldsymbol c\} / (pq\cup ac)$ or $conv\{\boldsymbol p, \boldsymbol b, \boldsymbol q, \boldsymbol d\}/(pq\cup bd)$. The intersections of four facets of $K_1$ which associate to $\boldsymbol p$ with $\lambda K_1+\boldsymbol u$ are two triangles and two quadrilaterals.

Case 3. $\boldsymbol u\in K_1/ (conv\{\boldsymbol p, \boldsymbol a, \boldsymbol q, \boldsymbol c\} \cup conv\{\boldsymbol p, \boldsymbol b, \boldsymbol q, \boldsymbol d\})$. The intersections of four facets of $K_1$ which associate to $\boldsymbol p$ with $\lambda K_1+\boldsymbol u$ are two quadrilaterals, one triangle and one pentagon.

\begin{lem} \label{070410}
If $\lambda K_1+X$ is an optimal $m$-covering configuration, and $\boldsymbol u\in X$ such that $\lambda K_1+\boldsymbol u$ covers a vertex $\boldsymbol v$ of $K_1$. Then the intersections of four facets of $K_1$ which associate to $\boldsymbol v$ with $\lambda K_1+\boldsymbol u$ are contained in the homothetic copies of the corresponding triangular facets with ratio not greater than $\lambda$, respectively.
\end{lem}

\begin{proof}
Without loss of generality, we can assume that $\boldsymbol v=\boldsymbol p$. Then the four facets of $K_1$ associated to $\boldsymbol p$ are $\Delta pab, \Delta pbc, \Delta pcd, \Delta pda$, respectively. Then by Observation and geometry computation, we know the four facets must intersect $\lambda K_1+\boldsymbol u$, and there are three possibilities.

\begin{figure}[htbp]
\centering
\includegraphics[scale=0.4]{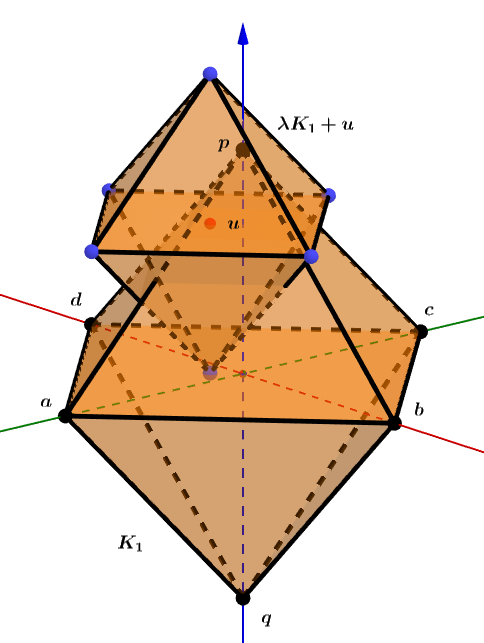}
\caption{$K_1$}
\end{figure}

Here we only give the proof for the third case, and the other two cases can be proved in a similar way. That is, $\boldsymbol u \in K_1/ (conv\{\boldsymbol p, \boldsymbol a, \boldsymbol q, \boldsymbol c\} \cup conv\{\boldsymbol p, \boldsymbol b, \boldsymbol q,\boldsymbol d\})$, see Figure $2$. The intersections of four facets of $K_1$ which associate to $\boldsymbol p$ with $\lambda K_1+\boldsymbol u$ are two quadrilaterals, one triangle and one pentagon, see Figure $3$.

\begin{figure}[htbp]
\centering
\includegraphics[scale=0.5]{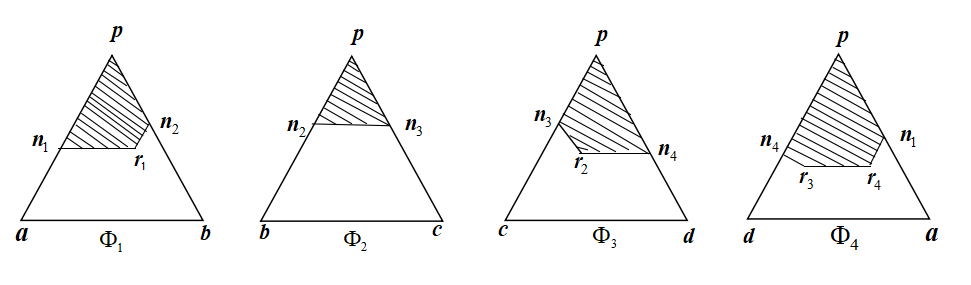}
\caption{$K_1$ intersects $\lambda K_1+\boldsymbol u$}
\end{figure}

As Figure.$3$, they are the four facets of $K_1$ which contain the vertex $\boldsymbol p$. Denote the four kinds of shadows by $\Phi_1, \Phi_2, \Phi_3, \Phi_4$ respectively. To emphasize their connection to the vertex $\boldsymbol p$, we also denote them by $\Phi_1(\boldsymbol p), \Phi_2(\boldsymbol p), \Phi_3(\boldsymbol p), \Omega_4(\boldsymbol p)$ respectively.
And about the regions of Figure $3$, since $\lambda K_1+\boldsymbol u$ is a translation of $\lambda K_1$, by routine calculation, we have the following results:

$$n_1r_1//ab// \lambda ab+\boldsymbol u, n_2n_3//bc// \lambda bc+\boldsymbol u,$$

$$r_2n_4//cd// \lambda cd+\boldsymbol u, r_3r_4//da// \lambda da+\boldsymbol u.$$

Since $\Phi_2(\boldsymbol p)\subseteq \Delta pbc\cap \lambda K_1+\boldsymbol u$ is a triangle, we only need to construct the triangles in $\Delta pab,\Delta pcd,\Delta pda$ that contain $\Phi_1(\boldsymbol p), \Phi_3(\boldsymbol p), \Phi_4(\boldsymbol p)$, respectively. Let $\boldsymbol r_1^{'}$ be the intersection of $L_{n_1r_1}$ with $pb$; $\boldsymbol r_2^{'}$ be the intersection of $L_{r_2n_4}$ with $pc$; $\boldsymbol r_3^{'}$ and $\boldsymbol r_4^{'},$ be the intersections of $L_{r_3r_4}$ with $pd, pa$, respectively. Then we have
$$\Phi_1(\boldsymbol p)\subset \Delta pn_1r_1^{'},\quad \Phi_2(\boldsymbol p)\subseteq \Delta pn_2n_3,$$
$$\Phi_3(\boldsymbol p)\subset \Delta pr_2^{'}n_4,\quad \Phi_4(\boldsymbol p)\subset \Delta pr_3^{'}r_4^{'}.$$
It is easy to see that these triangles are equilateral. Thus, $\Delta pn_1r_1^{'}, \Delta pn_2n_3$, $\Delta pr_2^{'}n_4,$ $\Delta pr_3^{'}r_4^{'}$ are homotheties of facets $\Delta pab,\Delta pbc,\Delta pcd,\Delta pda$, respectively. To prove the lemma it is sufficient to show that the ratios of these homotheties are not greater than $\lambda$. Let $\boldsymbol p^{'},\boldsymbol a^{'}, \boldsymbol b^{'}, \boldsymbol c^{'}, \boldsymbol d^{'}, \boldsymbol q^{'}$ be the vertices of $\lambda K_1+\boldsymbol u$ corresponding to vertices of $K_1$. We provide the proof in a few stages.

1. $\Delta pr_3^{'}r_4^{'}$ is a homothety of facet $\Delta pda$ with ratio not greater than $\lambda$.

Let $L_{\boldsymbol p^{'}}$ be a line that goes through $\boldsymbol p^{'}$ and parallels to segment $d^{'}a^{'}$ or $da$. And let $\boldsymbol m_{r_3^{'}r_4^{'}}, \boldsymbol m_{d^{'}a^{'}}, \boldsymbol m_{bc}$ be the middle points of $r_3^{'}r_4^{'}, d^{'}a^{'}, bc$, respectively. Let $\pi$ be the plane that contains $\Delta p^{'}d^{'}a^{'}$ and take $\boldsymbol p^{\pi}, \boldsymbol m_{r_3^{'}r_4^{'}}^{\pi}$ as the projection of $\boldsymbol p, \boldsymbol m_{r_3^{'}r_4^{'}}$ onto the plane $\pi$ in the direction of vector $\boldsymbol {m_{bc}p}.$ Then by elementary geometric properties,
 $$p^{\pi}m_{r_3^{'}r_4^{'}}^{\pi}//pm_{r_3^{'}r_4^{'}},
 |p^{\pi}m_{r_3^{'}r_4^{'}}^{\pi}|=|pm_{r_3^{'}r_4^{'}}|.$$
Since $\boldsymbol p\in \lambda K_1+\boldsymbol u$ and the opposite facets of $K_1$ are parallel to each other, $\boldsymbol p^{\pi}$ and $m_{r_3^{'}r_4^{'}}^{\pi}$ lie between the lines $L_{\boldsymbol p^{'}}$ and $L_{d^{'}a^{'}}$ in the plane $\pi$. $\Delta pr_3^{'}r_4^{'}$ is homothety of facets $\Delta pda$, so it is also a homothety of facet $p^{'}d^{'}a^{'}$. By the properties of resemble triangles, $|pm_{r_3^{'}r_4^{'}}|//p^{'}m_{d^{'}a^{'}}$. Then, we have $$p^{\pi}m_{r_3^{'}r_4^{'}}^{\pi}//p^{'}m_{d^{'}a^{'}}, |p^{\pi}m_{r_3^{'}r_4^{'}}^{\pi}|\leq |p^{'}m_{d^{'}a^{'}}|.$$
Thus,
\begin{equation}
|pm_{r_3^{'}r_4^{'}}|\leq |p^{'}m_{d^{'}a^{'}}|.\tag{1}
\end{equation}
Because $\Delta p^{'}d^{'}a^{'}$ is a homothety of facet $\Delta pda$ with ratio $\lambda$, combining with the properties of resemble triangles, $\Delta pr_3^{'}r_4^{'}$ is homothety of facet $\Delta pda$ and the ratio of homothety is less than $\lambda$.

2. $\Delta pn_1r_1^{'}$ is a homothety of facet $\Delta pab$ with ratio not more than $\lambda$; $\Delta pr_2{'}n_4$ is homothety of facet $\Delta pcd$ with ratio not more than $\lambda$.

Since $\lambda K_1+\boldsymbol u$ is a homothety of $K_1$, $p^{'}a^{'}//pn_1\subset pa$. $pn_1\subset \lambda K_1+\boldsymbol u$, then
\begin{equation}
|pn_1|\leq |p^{'}a^{'}|.\tag{2}
\end{equation}
Thus, $\Delta pn_1r_1^{'}$ is a homothety of facet $\Delta pab$ with ratio less than $\lambda$.

In $\Delta pr_2{'}n_4$, $p^{'}d^{'}//pn_4\subset pd$. Also, $pn_4\in \lambda K_1+\boldsymbol u$, then we have
\begin{equation}
|pn_4|\leq |p^{'}d^{'}|. \tag{3}
\end{equation}
Thus, $\Delta pr_2{'}n_4$ is a homothety of faces $\Delta pab$ with ratio less than $\lambda$.

3. $\Delta pn_2n_3$ is a homothety of facet $\Delta pbc$ with ratio not more than $\lambda$.

Similarly to Part 2, for $\Delta pn_2n_3$, $p^{'}b^{'}//pn_2\subset pb$. $pn_2\subset \lambda K_1+\boldsymbol u$, then we have
\begin{equation}
|pn_2|\leq |p^{'}b^{'}|.\tag{4}
\end{equation}
Thus, $\Delta pn_2n_3$ is a homothety of facet $\Delta pbc$ with ratio less than $\lambda$.

Therefore, the proof of the lemma is completed.
\end{proof}

\begin{rem}
The equalities of (1)(2)(3)(4) hold in the same time if and only if when $\boldsymbol p=\boldsymbol p^{'}$, in this case, the intersections of four facets of $K_1$ which associate to the vertex $\boldsymbol v$ with $\lambda K_1+\boldsymbol u$ are contained in the homothety of the corresponding facets with ratio $\lambda$, respectively; When $\boldsymbol p\in int(\lambda K_1+\boldsymbol u)$, then (1)(2)(3)(4) are strict inequalities; That is to say, the intersections of four faces of $K_1$ which associate to the vertex $\boldsymbol v$ with $\lambda K_1+\boldsymbol u$ are contained in the homothety of the corresponding facets with ratio less than $\lambda$, respectively. When $\boldsymbol p\in \partial (\lambda K_1+\boldsymbol u)/ \{\boldsymbol p^{'}\}$, the equalities of (1)(2)(3)(4) can not be hold in the same time.
\end{rem}

The following lemma is easy to prove, so we omit its proof here.
\begin{lem}
If $T$ is an equilateral triangle in the two-dimensional plane, then $\gamma_3(T)=\frac{2}{3}.$
\end{lem}


\begin{rem}
 It is easy to see that the optimal $3$-covering configuration of $T$ is unique, see Figure 4, and $\boldsymbol t_1$, $\boldsymbol t_2$, $\boldsymbol t_3$ are translative vectors, $\boldsymbol O$ is the centre of $T$.
\end{rem}
\begin{figure}[htbp]
\centering
\includegraphics[scale=0.3]{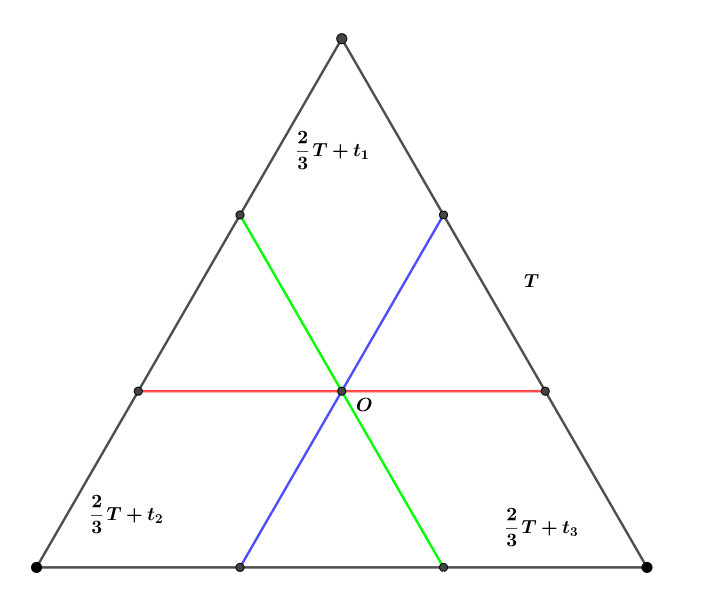}
\caption{The unique optimal $3$-covering configuration of $T$}
\end{figure}

\begin{lem}\textrm{(Zong, see\cite{Zong3} ) }
Let $K$ be a three-dimensional convex body, $\mu$ be a real number satisfying $0<\mu<1$, $R$ be a closed region on $\partial(K)$ with boundary $\Gamma$ and a relatively interior point $\boldsymbol I$. If $\Gamma\cup \{I\}\subset \mu K+\boldsymbol y$ holds for some point $\boldsymbol y$, then we have $R\subset \mu K+\boldsymbol y.$
\end{lem}

\begin{proof}[Proof of Theorem $1.2$]
We first prove $\gamma_6(K_1)\leq \frac{2}{3}.$ To show this, we take
$$\Gamma=\{\boldsymbol x=(x_1,x_2,x_3):x_1=\frac{1}{3}, \boldsymbol x\in \partial (K_1)\},$$
$\mu=\frac{2}{3}$, $\boldsymbol y=(\frac{1}{3},0,0)$, and let $R$ denote the part of $\partial (K_1)$ bounded by $\Gamma$ and containing $(1,0,0)$. By Lemma 3.6 we have $R\subset \mu K_1+\boldsymbol y$. Therefore, in this case $K_1$ can be covered by the union of $\mu K_1\pm (\frac{1}{3},0,0)$, $\mu K_1\pm (0,\frac{1}{3},0)$, $\mu K_1\pm (0,0,\frac{1}{3})$, and thus $\gamma_6(K_1)\leq \frac{2}{3}$.

On the other hand, if $\gamma_6(K_1)<\frac{2}{3}$. Since the number of vertices are the same as the number of homothetic copies of $K_1$, by Lemma 2.1 and Pigeonhole principle, to construct an optimal configuration, each homothetic copy should have and only have one vertex of $K_1$. In other words, there are six points $\boldsymbol u_i\in K_1,,(i=1,\dots,6)$ such that $\boldsymbol v_i\in \gamma_6(K_1)+\boldsymbol u_i$ and $K_1\subset \cup_{i=1}^6\gamma_6(K_1)+\boldsymbol u_i$. Then each facet of $K_1$ must intersect with three homothetic copies of $K_1$.

However, take $\Delta pab$ as an illustration, let $\boldsymbol u_1, \boldsymbol u_2, \boldsymbol u_3$ be the three points in $\mathbb{E}^3$ such that $\boldsymbol p\in\gamma_6K_1+\boldsymbol u_1$, $\boldsymbol a\in\gamma_6K_1+\boldsymbol u_2$, $\boldsymbol b\in\gamma_6K_1+\boldsymbol u_3.$ Then we have $\Delta pab\subset \cup_{i=1}^3\gamma_6K_1+\boldsymbol u_i$. By Lemma 3.1 and Remark 3.2, the intersections of $\Delta pab$ with the three homothetic copies are contained in an equilateral triangle with ratio not greater than $\gamma_6(K_1)$, respectively. It is a contradiction to Lemma 3.4. Hence we complete the proof of the theorem.
\end{proof}

\begin{cor}
The optimal $6$-covering configuration of $K_1$ is unique.
\end{cor}

\begin{proof}
By Theorem 3.7, $\gamma_6(K_1)=\frac{2}{3}$ and $\frac{2}{3}K_1+\{\pm(\frac{2}{3},0,0),\pm(0,\frac{2}{3},0), \pm(0,0,\frac{2}{3})\}$ is an optimal $6$-covering configuration. If $\frac{2}{3}K_1+X^{'}$ is also an optimal covering configuration, but $X^{'}\neq \{\pm(\frac{2}{3},0,0),\pm(0,\frac{2}{3},0), \pm(0,0,\frac{2}{3})\}$. By Remark 3.2 and Lemma 3.4, there is always a facet of $K_1$ which can not be covered completely. This completes the proof of the corollary.
\end{proof}

\begin{lem}
Let $\boldsymbol c_1, \boldsymbol c_2, \boldsymbol c_3$ be the centres of three facets of $K_1$ that intersect at only one vertex $\boldsymbol v$ for $K_1$, respectively. If $\boldsymbol c_1, \boldsymbol c_2, \boldsymbol c_3\in \lambda K_1+\boldsymbol u$, then $\lambda \geq \frac{2}{3}.$
\end{lem}

\begin{proof}
Without loss of generality, we can assume that $\boldsymbol v=\boldsymbol p$. By the symmetry of $K_1$, we let $\Delta pab, \Delta pbc, \Delta pcd$ be the three facets of the hypothesis and correspond to the centres $\boldsymbol c_1, \boldsymbol c_2, \boldsymbol c_3$. Since $\boldsymbol c_1, \boldsymbol c_2, \boldsymbol c_3\in \lambda K_1+\boldsymbol u$, by the convexity of $K_1$,
$$conv\{\boldsymbol c_1, \boldsymbol c_2, \boldsymbol c_3\}=\Delta c_1c_2c_3\subset \lambda K_1+\boldsymbol u. $$
Moreover, $\boldsymbol c_1$, $\boldsymbol c_2$,$\boldsymbol c_3$ are centres of facets $\Delta pab, \Delta pbc, \Delta pcd$, so the plane $\pi$ that contains $\Delta c_1c_2c_3$ is parallel to the plane $\beta$ that contains base square $abcd$. Then $\Delta c_1c_2c_3$ is contained in a square $S\subset \lambda K_1+\boldsymbol u$, where $S$ is a homothety of base square $abcd$ with ratio $\lambda^{'}$. That is to say,
$$\Delta c_1c_2c_3\subset S\subset \lambda K_1+\boldsymbol u,$$
hence $\lambda^{'}\leq \lambda$. By elementary geometry property, the smallest homothety of base square $abcd$ that contains $\Delta c_1c_2c_3$ is unique, and we get that the smallest ratio of homothety of base square $abcd$ that contains $\Delta c_1c_2c_3$ is $\frac{2}{3}$. Hence $\frac{2}{3}\leq \lambda^{'}\leq \lambda$. This completes the proof of the lemma.
\end{proof}

\begin{cor}
Let $\boldsymbol c_1, \boldsymbol c_2$ be the centres of two facets of $K_1$ that intersect at only one vertex $\boldsymbol v$ for $K_1$, respectively. If $\boldsymbol c_1, \boldsymbol c_2\in \lambda K_1+\boldsymbol u$, then $\lambda \geq \frac{2}{3}.$
\end{cor}

\begin{lem}
Let $\boldsymbol c_1, \boldsymbol c_2$ be the centres of two facets of $K_1$ that do not intersect. If $\boldsymbol c_1, \boldsymbol c_2\in \lambda K_1+\boldsymbol u$, then $\lambda\geq 1.$
\end{lem}

\begin{proof}
Without loss of generality, let $\Delta pab, \Delta qcd$ be the two facets of $K_1$ that satisfies the above assumption. Let $\boldsymbol m_{ab}, \boldsymbol m_{cd}$ be the middle points of the segments $ab, cd$, then the segment $c_1c_2$ lies in quadrilateral $pm_{ab}qm_{cd}$. Since the convexity of $K_1$, the segment $c_1c_2\in \lambda K_1+\boldsymbol u$. Moreover, $\lambda K_1+\boldsymbol u$ is a homothety of $K_1$, there is a quadrilateral $Q\in \lambda K_1+\boldsymbol u$ containing segment $c_1c_2$ and being a homothety of quadrilateral $pm_{ab}qm_{cd}$. Same as the proof of Lemma 3.9, by elementary geometry property, the smallest homothety of $pm_{ab}qm_{cd}$ with $c_1c_2$ in it is $pm_{ab}qm_{cd}$ itself. That is to say, $\lambda \geq 1$. Then we complete the proof of lemma.
\end{proof}

\begin{cor}
If $\lambda K_1+\boldsymbol u$ contains three centres of facets of $K_1$, then $\lambda$ is not less than $\frac{2}{3}.$
\end{cor}

\begin{proof}[Proof of Theorem $1.3$]
Since the sequence $\gamma_m(K_1)$ is nonincreasing,
$$\gamma_9(K_1)\leq \gamma_8(K_1)\leq \gamma_7(K_1) \leq \gamma_6(K_1)\leq \frac{2}{3}.$$

In the other side, suppose the contrary, name that $\gamma_9(K_1) < \frac{2}{3}$. By the definition of $\gamma_m(K_1)$, we can assume that $K_1\subset \cup_{i=1}^{9}\gamma_9K_1+\boldsymbol u_i$. Then the vertices and facets of $K_1$ are completely covered by these homothetic copies. By Lemma 2.1, there are always six homothetic copies of $K_1$, each of them contains one vertex of $K_1$. For the sake of narrative, in the proof of this theorem we are going to let $\boldsymbol v_i, (i\in \{1,2,\dots,6\})$ represent the vertex of $K_1$ for the moment. Then we can assume that $\boldsymbol v_i\in \gamma_9K_1+\boldsymbol u_i, (i\in\{1,2,\dots,6\})$. Since the assumption $\gamma_9(K_1)<\frac{2}{3}=\gamma_6(K_1)$, by Lemma 3.2 and Theorem 3.7, eight facets of $K_1$ can not be completely covered by these six homothetic copies. Furthermore, by the proof of lemma 3.4, the eight centres of facets of $K_1$ can not be covered. In this case, the last three homothetic copies of $K_1$ have to contain these eight centres. Thus, by Corollary 3.12 and the Pigeonhole principle, there is always a homothetic copy of $K_1$ containing three centres of facets of $K_1$. Therefore, the ratio of homothetic copy $\gamma_9K_1+\boldsymbol u_i$ can not be less than $\frac{2}{3}$. This is a contradiction to the hypothesis. Thus, $$\gamma_9(K_1)=\frac{2}{3}.$$
This implies
$\gamma_{9}(K_1)=\gamma_{8}(K_1)=\gamma_{7}(K_1) =\gamma_{6}(K_1)=\frac{2}{3}.$

\end{proof}

\section{Case for $m=10, 11, 12, 13.$ }

\begin{lem}
Let $\lambda K_1+\boldsymbol u$ be a homothetic copy of $K_1$ containing three $\eta$-node points from three different facets of $K_1$, then $\lambda\geq \eta.$
\end{lem}

\begin{proof}
It is easy to see that to prove this lemma we only need to find the the smallest ratio of homothetic copy of $K_1$ that satisfies the above condition. Clearly, the smallest case is when the three facets have one vertex in common. Let $F_1, F_2, F_3$ be three facets of $K_1$ with one vertex $\boldsymbol p$ of $K_1$ in common, and let $\boldsymbol e, \boldsymbol f, \boldsymbol g$ be three $\eta$-node points in facet $F_1, F_2, F_3$ respectively, and they are three $\eta$-node points close to $\boldsymbol p$. By elementary calculation, $\boldsymbol e, \boldsymbol f, \boldsymbol g$ must be on the boundary of this smallest homothetic copy of $K_1$. By the convexity of $K_1$, the convex hull $\Delta efg$ of $\{\boldsymbol e, \boldsymbol f, \boldsymbol g\}$ must be contained in the smallest homothetic copy of $K_1$. The plane that contains $\Delta efg$ is parallel to the base square $abcd$, so $\Delta efg$ is contained in a homothety of $abcd$. By the simple calculation, the ratio of the smallest homothety of quadrilateral $abcd$ that contains $\Delta efg$ is $\eta$. This implies the ratio of the smallest homothetic copy of $K_1$ that contains $\boldsymbol e, \boldsymbol f, \boldsymbol g$ is $\eta$. Thus, $\lambda\geq \eta.$
\end{proof}

\begin{proof}[Proof of Theorem $1.4$]
By Lemma 3.6, we first take
$$\Gamma_1=\{\boldsymbol x=(x_1,x_2,x_3):x_1=\frac{2}{5}, \boldsymbol x\in \partial (K_1)\},$$
$$\Gamma_2=\{\boldsymbol x=(x_1,x_2,x_3):x_1\in (\frac{1}{5}, \frac{4}{5}), x_2=\frac{1}{5} \text{\quad or}\quad x_1=\frac{1}{5}, x_2\in (\frac{1}{5}, \frac{4}{5}), \boldsymbol x\in \partial (K_1)\},$$
$\mu=\frac{3}{5}$, $\boldsymbol y=(\frac{2}{5},0,0), \boldsymbol y^{'}=(\frac{2}{5},\frac{2}{5},0)$, and let $R_1, R_2$ denote the part of $\partial (K_1)$ bounded by $\Gamma_1, \Gamma_2$ and containing $(1,0,0), (\frac{1}{2}, \frac{1}{2}, 0)$, respectively. By Lemma 3.6 we have $R_1\subset \mu K_1+y, R_2\subset \mu K_1+y^{'}$. Therefore, in this case, $K_1$ can be covered by the union of $\mu K_1\pm (\frac{2}{5},0,0)$, $\mu K_1\pm (0,\frac{2}{5},0)$, $\mu K_1\pm (0,0,\frac{2}{5})$, $\mu K_1\pm (\frac{2}{5},\frac{2}{5},0)$, $\mu K_1\pm (0,\frac{2}{5},\frac{2}{5})$, and thus $\gamma_{10}(K_1)\leq \frac{3}{5}$.

In the other side, we assume that $\gamma_{10}(K_1)<\frac{3}{5}$. By Lemma 2.1, we know that there are always six homothetic copies of $K_1$, each of them contains one vertex of $K_1$. Let $\boldsymbol v_i, i\in\{1,2,\dots,6\}$ be vertices of $K_1$, $\boldsymbol u_j, (j\in\{1,2,\dots,10\})$ such that $K_1\subseteq_{j=1}^{10}\gamma_{10}K_1+\boldsymbol u_j$. Without loss the generality, we let $\boldsymbol v_i\in \gamma_{10}K_1+\boldsymbol u_j$, $(i,j\in \{1,2, \dots, 6\})$. Since $\gamma_{10}(K_1)<\frac{3}{5}<\frac{2}{3}$, $K_1$ can not be covered completely by these six homothetic copies of $K_1$. And by Lemma 3.4, all facets of $K_1$ can not be covered completely. Furthermore, by calculation, every centre and $\gamma_{10}(K_1)$-node points in each facet of $K_1$ must be covered by the last four homothetic copies. By the convexity of $K_1$, if one homothetic copy contains all $\gamma_{10}(K_1)$-node points from one facet of $K_1$, then it must contain the centre of this facet. By Lemma 4.1 and the Pigeonhole principle, for the last four homothetic copies of $K_1$, each of them can not contain all $\gamma_{10}(K_1)$-node points from three adjacent facets, but each of them must contain all $\gamma_{10}(K_1)$-node points from two adjacent facets at least.

\begin{figure}[htbp]
\centering
\includegraphics[scale=0.5]{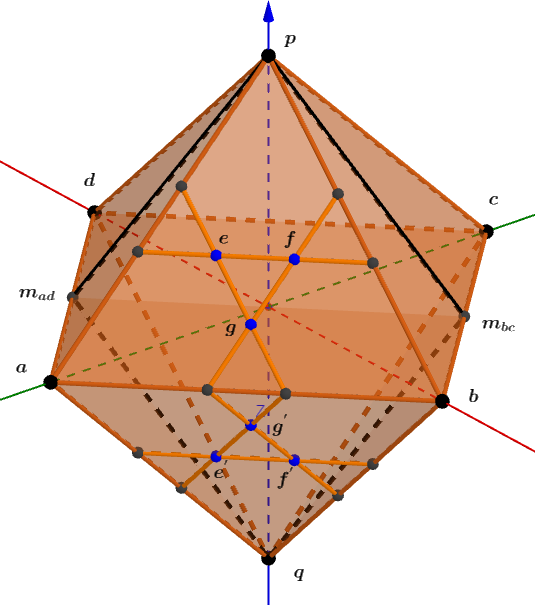}
\caption{}
\end{figure}

Next, we will show that if one homothetic copy of $K_1$ covers all $\frac{3}{5}$-node points from two adjacent faces, then the ratio of this homothetic copy is not less than $\frac{3}{5}$. Without loss of generality, we can let $\Delta pab, \Delta qab$ be the two adjacent faces of $K_1$ and $\boldsymbol e, \boldsymbol f, \boldsymbol g\in \Delta pab, \boldsymbol e^{'}, \boldsymbol f^{'}, \boldsymbol g^{'}\in \Delta qab$ be $\frac{3}{5}$-node points in these adjacent faces, see Figure $5$. However, it is sufficient to prove the ratio of a homothetic copy of $K_1$ covering $\frac{3}{5}$-node points $\boldsymbol e, \boldsymbol f, \boldsymbol e^{'}, \boldsymbol f^{'}$ is not less than $\frac{3}{5}$. Let $\lambda K_1+\boldsymbol u$ be one homothetic copy containing $\frac{3}{5}$-node points $\boldsymbol e, \boldsymbol f, \boldsymbol e^{'}, \boldsymbol f^{'}$. By the convexity of $K_1$, this homothetic copy contains the quadrilateral $efe^{'}f^{'}$. Let $\boldsymbol m_{ad}, \boldsymbol m_{bc}$ be the middle points of segment $ad, bc$, respectively. By elementary geometric property, the plane that contains $efe^{'}f^{'}$ is parallel to the plane that contains $pm_{ad}qm_{bc}$. Since $efe^{'}f^{'}\subset \lambda K_1+\boldsymbol u$, there is a homothety of $pm_{ad}qm_{bc}$ containing $efe^{'}f^{'}$. By elementary geometry argument, we get that the smallest ratio of homothety of $pm_{ad}qm_{bc}$ containing $efe^{'}f^{'}$ is $\frac{3}{5}$. It implies that $\lambda\geq \frac{3}{5}$. Hence, we have proved that $\gamma_{10}(K_1)\geq \frac{3}{5}$. This is a contradiction to the assumption. So we complete the proof of this theorem.
\end{proof}

\begin{proof}[Proof of Theorem $1.5$]
Since the sequence $\gamma_m(K_1)$ is nonincreasing,
$$\gamma_{13}(K_1)\leq \gamma_{12}(K_1) \leq \gamma_{11}(K_1) \leq \gamma_{10}=\frac{3}{5}.$$

In the other hand, if $\gamma_{13}(K_1)<\frac{3}{5}$, then by the proof of Theorem 4.2, we can see that the first six homothetic copies of $K_1$ which contain the vertices of $K_1$ can not cover the $K_1$ completely. Also, they can not cover the $\frac{3}{5}$-node points of each facet of $K_1$. Since there are seven homothetic copies of $K_1$ left and $K_1$ has eight facets, by the Pigeonhole principle and Lemma 4.1, there must be one homothetic copy containing $\frac{3}{5}$-node points from two facets at least. By the proof of Theorem 4.2, if a homothetic copy of $K_1$ covers all $\frac{3}{5}$-node points from two adjacent facets, then the ratio of this homothetic copy is not less than $\frac{3}{5}$, i.e., $$\gamma_{13}(K_1)\geq \frac{3}{5}.$$
This is implies that $\gamma_{13}(K_1)= \frac{3}{5}$. Thus,
$\gamma_{13}(K_1)=\gamma_{12}(K_1)=\gamma_{11}(K_1) =\gamma_{10}=\frac{3}{5}.$
\end{proof}

\section{Case for $m=14, 15, 16, 17.$ }

\begin{proof}[Proof of Theorem $1.6$]
By lemma 3.6, we first take
$$\Gamma_1=\{\boldsymbol x=(x_1,x_2,x_3):x_1=\frac{3}{7}, \boldsymbol x\in \partial (K_1)\},$$
$$\Gamma_2=\{\boldsymbol x=(x_1,x_2,x_3):x_1\in (\frac{2}{7}, \frac{6}{7}), x_2=\frac{1}{7} \text{\quad or}\quad x_1=\frac{2}{7}, x_2\in (\frac{1}{7}, \frac{5}{7}), \boldsymbol x\in \partial (K_1)\},$$
$$\Gamma_3=\{\boldsymbol x=(x_1,x_2,x_3):x_1\in (\frac{1}{7}, \frac{5}{7}), x_2=\frac{2}{7} \text{\quad or}\quad x_1=\frac{1}{7}, x_2\in (\frac{2}{7}, \frac{6}{7}), \boldsymbol x\in \partial (K_1)\},$$
$\mu=\frac{4}{7}$, $\boldsymbol y=(\frac{3}{7},0,0), \boldsymbol y^{'}=(\frac{2}{7},\frac{1}{7},0), \boldsymbol y^{"}=(\frac{1}{7},\frac{2}{7},0)$, and let $R_1, R_2, R_3$ denote the part of $\partial (K_1)$ bounded by $\Gamma_1, \Gamma_2, \Gamma_3$ and containing $(1,0,0), (\frac{4}{7}, \frac{3}{7}, 0), (\frac{3}{7}, \frac{4}{7}, 0)$, respectively. By Lemma 3.6 we have $R_1\subset \mu K_1+\boldsymbol y, R_2\subset \mu K_1+\boldsymbol y^{'}, R_3\subset \mu K_1+\boldsymbol y^{''}$. Therefore, in this case, $K_1$ can be covered by the union of $\mu K_1\pm (\frac{3}{7},0,0)$, $\mu K_1\pm (0,\frac{3}{7},0)$, $\mu K_1\pm (0,0,\frac{3}{7})$, $\mu K_1\pm (\frac{2}{7},\frac{1}{7},0)$, $\mu K_1\pm (\frac{1}{7},\frac{2}{7},0)$, and thus $\gamma_{14}(K_1)\leq \frac{4}{7}$.

In the other hand, If $\gamma_{14}(K_1)<\frac{4}{7}$, by the proof of Theorem 4.2, except the first six homothetic copies of $K_1$ that cover vertices of $K_1$, there are eight homothetic copies of $K_1$ left. By Lemma 4.1, each of them only can cover $\frac{4}{7}$-node points from two adjacent facets of $K_1$ at most, but must cover three $\frac{4}{7}$-node points at least. Without loss of generality, similarly to Theorem 4.2, we choose adjacent facets $\Delta pab, \Delta qab$. In particular, the line through $\boldsymbol e$ and $\boldsymbol f$ is parallel to $ab$. $\boldsymbol e, \boldsymbol f, \boldsymbol g\in \Delta pab, \boldsymbol e^{'}, \boldsymbol f^{'}, \boldsymbol g^{'}\in \Delta qab$ be the $\frac{4}{7}$-node points from these two adjacent faces. And let $\boldsymbol c, \boldsymbol c^{'}$ be the middle points of segments $ef, e^{'}f^{'}$. We find if a homothetic copy of $K_1$ covers $\boldsymbol e, \boldsymbol f, \boldsymbol e^{'}, \boldsymbol f^{'}$, then the ratio of the smallest possible homothetic copy of $K_1$ that satisfies above condition is $\frac{5}{7}$. Thus, these $\frac{4}{7}$-node points from these two adjacent facets must be covered by two homothetic copies $K_1$. Let $\lambda K_1+\boldsymbol u_1, \lambda K_1+\boldsymbol u_2$ be two homothetic copies of $K_1$ that cover $\frac{4}{7}$-node points from $\Delta pab$ and $\Delta qab$. In the following, we will show that if the homothetic copy $\lambda K_1+\boldsymbol u_1$ or $\lambda K_1+\boldsymbol u_2$ covers three $\frac{4}{7}$-node points from two adjacent facets, then $\lambda\geq \frac{4}{7}.$

(i) $\lambda K_1+\boldsymbol u_1$ covers three $\frac{4}{7}$-node points only from one facet. Without loss of generality, we assume that $\lambda K_1+u_1$ contains three $\frac{4}{7}$-node points from $\Delta pab$. Then by the convexity of $K_1$, $\lambda K_1+\boldsymbol u$ contains the convex hull of these three points, i.e., $\Delta efg\subset \lambda K_1+\boldsymbol u$. Moreover, the plane that contains $\Delta efg$ is parallel to the plane that contains $\Delta pab$, then $\Delta efg$ must be contained in a homothety of $\Delta pab$. By simple calculation, the smallest ratio of homothety of $\Delta pab$ that contains $\Delta efg$ is $\frac{4}{7}$. This implies that $\lambda\geq \frac{4}{7}$.

(ii) $\lambda K_1+\boldsymbol u$ covers three $\frac{4}{7}$-node points from two facets.

Case 1: $\lambda K_1+\boldsymbol u$ covers arbitrarily three points from $\{\boldsymbol e, \boldsymbol f, \boldsymbol e^{'}, \boldsymbol f^{'}\}$, without loss of generality, we can assume that $\{\boldsymbol e, \boldsymbol f, \boldsymbol e^{'}\}\in \lambda K_1+\boldsymbol u_1$. By the convexity of $K_1$, $\Delta efe^{'}\subset \lambda K_1+\boldsymbol u_1$. Since the plane that contains $\Delta efe^{'}$ is parallel to the plane that contains $pm_{ad}qm_{bc}$, there must be a homothety of $pm_{ad}qm_{bc}$ such that $\Delta efe^{'}$ containing in it. By basic calculation, the ratio of the smallest possible homothetic copy of $K_1$ that satisfies above condition is $\frac{5}{7}$. Thus, $\lambda\geq \frac{5}{7}>\frac{4}{7}.$

Case 2: $\boldsymbol e, \boldsymbol g, \boldsymbol e^{'}\in \lambda K_1+\boldsymbol u_1$, then $\boldsymbol f, \boldsymbol f^{'}, \boldsymbol g^{'}\in \lambda K_1+\boldsymbol u_2$. Since the triangles $\Delta efg$, $\Delta e^{'}f^{'}g^{'}$ completely covered by these two homothetic copies of $K_1$, $\boldsymbol c, \boldsymbol c^{'}$ must be covered by one of the homothetic copies.

Subcase 1: $\boldsymbol c, \boldsymbol c^{'}\in \lambda K_1+\boldsymbol u_1$. Then $ecc^{'}e^{'}\in \lambda K_1+\boldsymbol u_1$. Since the plane that contains $ecc^{'}e^{'}$ is parallels to the plane that contains $pm_{ad}qm_{bc}$, there must be a homothety of $pm_{ad}qm_{bc}$ containing $ecc^{'}e^{'}$ in it. By elementary geometry calculation, the smallest ratio of homothety of $pm_{ad}qm_{bc}$ containing $\Delta ege^{'}$ is $\frac{4}{7}$. Thus, $\lambda\geq \frac{4}{7}.$

Subcase 2: $\boldsymbol c\in \lambda K_2+\boldsymbol u_1, \boldsymbol c^{'}\in \lambda K_1+\boldsymbol u_2$, then $\Delta ece^{'}\in \lambda K_1+\boldsymbol u_1$. Since the plane that contains $\Delta ece^{'}$ is parallel to the plane that contains $pm_{ad}qm_{bc}$, there is a homothety of $pm_{ad}qm_{bc}$ such that $\Delta ece^{'}$ containing in it. By simple calculation, the smallest ratio of homothety of $pm_{ad}qm_{bc}$ containing $\Delta ece^{'}$ is $\frac{4}{7}$. Thus, $\lambda \geq \frac{5}{7}>\frac{4}{7}.$

Case 3: $\boldsymbol e, \boldsymbol f, \boldsymbol g^{'}\in \lambda K_1+\boldsymbol u_1$, then $\boldsymbol e^{'}, \boldsymbol f^{'}, \boldsymbol g\in \lambda K_1+\boldsymbol u_2$. Since the triangles $\Delta efg$, $\Delta e^{'}f^{'}g^{'}$ are completely covered by these two homothetic copy of $K_1$, $\boldsymbol c, \boldsymbol c^{'}$ must be covered by one of the homothetic copies.

Subcase 1: $\boldsymbol c, \boldsymbol c^{'}\in \lambda K_1+\boldsymbol u_1,$ then  $ecc^{'}e^{'}\in \lambda K_1+\boldsymbol u_1$. This case is the same as the first subcase in Case 2. So, $\lambda\geq \frac{5}{7}>\frac{4}{7}.$

Subcase 2: $\boldsymbol c\in \lambda K_1+\boldsymbol u_1,$ then $\boldsymbol c^{'}\in \lambda K_1+\boldsymbol u_2.$ Then $\Delta ecg^{'}\in \lambda K_1+\boldsymbol u_1$. Since the plane that contains $\Delta ecg^{'}$ is parallel to $pm_{ad}qm_{bc}$, there is a homothety of $pm_{ad}qm_{bc}$ such that $\Delta ecg^{'}$ containing in it. By simple calculation, the smallest ratio of homothety of $pm_{ad}qm_{bc}$ containing $\Delta ecg^{'}$ is $\frac{4}{7}$, i.e., $\lambda\geq \frac{4}{7}.$ Thus, this is a contradiction to $\gamma_{14}(K_1)<\frac{4}{7}$. So, $\gamma_{14}(K_1)=\frac{4}{7}$.
\end{proof}

\begin{proof}[Proof of Theorem $1.7$]
Since the sequence $\gamma_m(K_1)$ is nonincreasing,
$$\gamma_{17}(K_1)\leq \gamma_{16}(K_1) \leq \gamma_{15}(K_1) \leq \gamma_{14}=\frac{4}{7}.$$

In the other hand, if $\gamma_{17}(K_1)<\frac{3}{5}$, then by the proof of Theorem 4.2, we can see that the first six homothetic copies of $K_1$ which contain the vertices of $K_1$ can not cover $K_1$ completely. Also, they can not cover the $\frac{4}{7}$-node points of each facet of $K_1$. Since there are eleven homothetic copies of $K_1$ left and $K_1$ has eight facets, by the Pigeonhole principle and Lemma 4.1, there must be one homothetic copy containing three $\frac{4}{7}$-node points at least. By the proof of Theorem 5.1, if a homothetic copy of $K_1$ covers three $\frac{4}{7}$-node points, then the ratio of this homothetic copy is not less than $\frac{4}{7}$. Thus, $$\gamma_{17}(K_1)\geq \frac{4}{7}.$$
This implies that $\gamma_{17}(K_1)= \frac{4}{7}$. Thus,
$$\gamma_{17}(K_1)=\gamma_{16}(K_1)=\gamma_{15}(K_1) =\gamma_{14}=\frac{4}{7}.$$

\end{proof}

 {\bf Acknowledgement:}  This work is supported by the National Natural Science Foundation of China (NSFC11921001) and the National Key Research and Development Program of China (2018YFA0704701). The authors are grateful to professor C.Zong for his supervision and discussion.

\end{document}